\documentclass[11pt, reqno]{amsart}

\usepackage{amsmath, amssymb}
\usepackage[all]{xypic}
\usepackage{epsfig}

\newcommand{\eeq}{\end{equation}}

\newcommand{\ZZ}{\mathbb{Z}}

\numberwithin{equation}{section}

\newtheorem{thm}{Theorem}[section]

\newtheorem{cor}[thm]{Corollary}

\def\slfrac#1#2{\hbox{\kern.1em %
 \raise.5ex\hbox{\the\scriptfont0 #1}\kern-.11em %
 /\kern-.15em\lower.25ex\hbox{\the\scriptfont0 #2}}}

\newcommand{\sF}{{\mathcal F}}
\newcommand{\sG}{{\mathcal G}}

\newcommand{\sgn}{\text{\rm sgn\,}}
\title[The de la Vall\'ee Poussin kernal]{The $L^{\!1}$ norms of de la Vall\'ee Poussin  kernels}
\author{Harsh Mehta}
\address{Department of Mathematics\\
University of Michigan\\
530 Church St.\\
Ann Arbor, MI 48109--1043  (USA)}

\email{hmehta@umich.edu}



\begin{document}

\begin{abstract}
Charles de la Vall\'ee Poussin defined two different kernels that bear his name.  This paper considers the ones
 are a linear combinations of two Fej\'er kernels, 
which are known as the delayed means. We  show that
the $L^1$ norms are constant in families of delayed means, and 
determine the exact value 
of the $L^{\!1}$ norm for some of them.

\end{abstract}

\date{November 2, 2013}
\maketitle

\section{Introduction}\label{S:Intro}

\noindent
This paper studies properties of certain summability kernels for Fourier series,
the de la Vall\'{e}e Poussin kernels defined below.
 Let $f(x)$ be a periodic function on $\mathbb{R}$ of period $1$, 
of finite $L^1$-norm in $[0,1]$, with Fourier series denoted
$$ 
f(x) \sim \sum_{k \in \ZZ} \hat{f}(k) e(kx),
$$
where $e(x) := e^{2 \pi i x}.$
 Define the partial sum
  of its Fourier series
\begin{equation}\label{E:sum}
S_n(f, x) :=\sum^{n}_{k=-n}\widehat f(k)e(kx)\\
\end{equation}
The Fej\'{e}r mean (of period $1$) with parameter $n$ is 
\begin{equation}\label{E:fejer}
\sigma_n(f, x) :=\frac{1}{n+1}\sum_{k=0}^{n}S_k(f, x)=\sum_{k=-n}^n\left(1-\frac{|k|}{n+1}\right)\widehat f(k)e(kx).
\end{equation}
This mean is given as singular integral of convolution type 
\[
\sigma_n(f, x) = \int_{0}^1 f(x-u) K_{n}(u) du,
\]
where the {\em Fej\'{e}r kernel} $K_n(x)$, rescaled for functions of period $1$,   is
\begin{equation}\label{Fejer}
K_n(x) := \Delta_{n+1}(x) = \frac{1}{n+1} \Big(\frac{\sin \pi (n+1)x}{\sin \pi x}\Big)^{\, 2}.
\end{equation}
In 1918 de la  Vall\'ee Poussin \cite{deVP18} introduced the {\em delayed means} 
(also called  {\em de la Vall\'{e}e Poussin sums} (\cite{Ef59},  \cite{Ser11}), as
\begin{equation}\label{E:originaldef}
\sigma_{n,p}(f, x) :=\frac{1}{p}\sum_{k=n}^{n+p-1} S_n(f, x)=\frac{n+p}{p}\sigma_{n+p-1}(f, x)-\frac{n}{p}\sigma_{n-1}(f, x),
\end{equation}
Here we follow the  notation for these means used in Zygmund \cite[Chap. III.1, (1.30)]{Zyg68}, after rescaling the periodicity interval from
$[0, 2\pi]$ to $[0,1].$
The means above are given by  the convolution integral  
\begin{equation}
\sigma_{n, p} (f, x) = \int_{0}^1 f(x-u) K_{n, p} (u) du,
\end{equation}
in which  $K_{n, p}(x)$ is the {\em de la Vall\'{e}e Poussin kernel function} \cite{deVP18} with parameters $(n,p)$.
This is found to be 
\begin{eqnarray}\label{E:deVPdef}
K_{n, p}(x) & =  & \frac{n+p}{p} \Delta_{n+p}(t) - \frac{n}{p} \Delta_n(t) \\
&=& 
\frac{1}{p}\Big(
\frac{(\sin\pi (n+p)x)^2 - (\sin \pi n x)^2}{(\sin \pi x)^2} 
\label{E:deVPdef2}
\Big).
\end{eqnarray}
Taking $n=0$ and $p=n$ we obtain $K_{0, n}(x) = \Delta_{n}(x) = K_{n-1}(x)$, the Fej\'{e}r kernel
with a shifted parameter.
All of these kernels  have
\begin{equation}\label{norm}
\int_{0}^1 K_{n, p}(x) dx = 1.
\end{equation}


For each parameter set $(n,p)$ there is a family 
 $\sF_{n, p} := \{ K_{nN, pN}(x) : N \ge 1\}$ of kernel functions
indexed by  the positive integer parameter $N \ge 1$.
This family  forms a {\em summation kernel}
in the sense of Walker \cite[(8.1), (8.23)]{Wal88}, or a {\em finite $\theta$-factor} 
in the sense of Butzer and Nessel \cite[Sec. 1.2.5]{BN71}. 
These authors  define  a general {\em finite $\theta$-factor} to be an infinite family of data $\{ \theta_N(j): N\ge 1, j \in \mathbb{Z}\}$ with 
\[
S_{\theta}(f, x) = \sum_{j= - m(N)}^{m(N)} \theta_N(j) \hat{f}(j) e( jx)
\]
and with a  function $m(N) \to \infty$ as $N \to \infty$.
For  the de la Vall\'{e}e Poussin kernel with parameters $(n, p)$ 
the associated $\theta$-factor  takes  $m(N) = N$
and sets 
\[ 
\theta_N(j) := v_{n, p}( \frac{j}{N}),
\]
where $v_{n, p}(x)$ being a compactly supported piecewise linear function on $\mathbb{R}$
given by
\[
v_{n, p}(u) :=  \left\{
\begin{array}{cl}
1  & \mbox{if} ~~|u| \le n, \\
\frac{ n + p- |u|}{p}   & \mbox{if}~~ n \le |u| \le n+p,\\
0 &  \mbox{if} ~~~|u| \ge n+p
\end{array}
\right.
\]
Special cases of  function are pictured in \figurename~\ref{v(1,1)} 
and \figurename~\ref{v(2,2)}.

\begin{figure}[h]
\centering
\includegraphics[width=90mm]{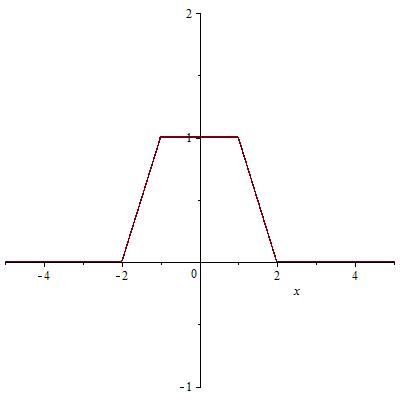}
\caption{ $ v_{n,p}(x)$ with $(n, p)= (1,1)$, so $N=1$ and $(r, s) = (1, 2)$.}
\label{v(1,1)}
\end{figure}

\begin{figure}[h]
\centering
\includegraphics[width=90mm]{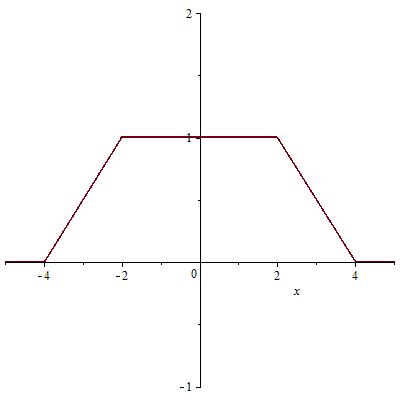}
\caption{ $v_{n,p}(x)$ with $(n,p)=(2,2)$, so $N=2$ and  $(r, s) = ( 1, 2)$.}
\label{v(2,2)}
\end{figure}

For later convenience we make a linear change of variables in the parameters,
setting $N = \gcd(n, n+p)$ and writing
 $Nr=n$ and $Ns=n+p$, so that $0 \leq r < s$and  $r$ and $s$ are relatively prime.
(Thus  $n=Nr$ and $p= N(s-r)$.)
In the new parameters  the
de la Vall\'{e}e Poussin kernels \eqref{E:deVPdef} become
\begin{equation}\label{E:genVdef0}
V_{rN,sN}(x) := K_{rN, (s-r)N}(x)  = \frac{s\Delta_{sN}(x) - r\Delta_{rN}(s)}{s-r}.
\end{equation}
In consequence
\begin{align}
V_{rN,sN}(x)&=\frac{(\sin sN\pi x)^2 - (\sin rN\pi x)^2}{(s-r)N(\sin \pi x)^2} \label{E:genVdef2} \\
&= \sum_{n=-sN+1}^{sN-1} v_{r,s-r}(n/N)e(nx)\label{E:genVdef3}
\end{align}
where for real $u$
\begin{equation}\label{E:vdef}
 v_{r,s-r}(u) = \begin{cases}
         1 & (|u|\le r), \\
         \displaystyle\frac{s-|u|}{s-r} & (r\le |u|\le s), \\
         0 & (|u|>s).
       \end{cases}
\end{equation}
A particularly well known case of his summability kernel \cite[Sect. 29]{deVP19}
occurs for $r=1, s=2$, and is 
\begin{align}\label{E:Vdef1}
K_{N,N}(x) = V_{N, 2N}(x) &= 2\Delta_{2N}(x) - \Delta_N(x) \\
 &= \frac{(\sin 2N\pi x)^2 - (\sin N\pi x)^2}{N(\sin\pi x)^2} \label{E:Vdef2} \\
       &=\sum_{|n|\le N} e(nx) \ + \ \sum_{N<|n|<2N} \big(2-|n|/N\big)e(nx)\,.
       \label{E:Vdef3}
\end{align}
The  shown for $N=1$ in \figurename~\ref{V_{1,2}}
and for $N=2$ in \figurename~\ref{V_{2,4}}.

\begin{figure}[h]
\centering
\includegraphics[width=90mm]{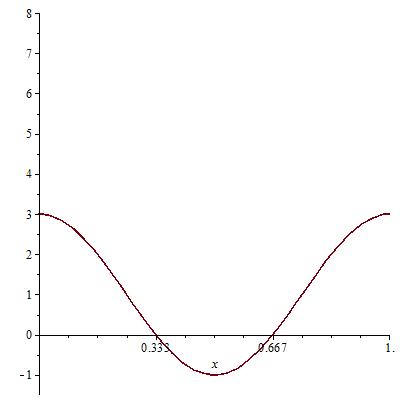}
\caption{de la Vall\'{e}e Poussin kernel $K_{1,1}(x) = V_{1,2}(x)$ on $[0,1]$.}
\label{V_{1,2}}
\end{figure}

\begin{figure}[h]
\centering
\includegraphics[width=90mm]{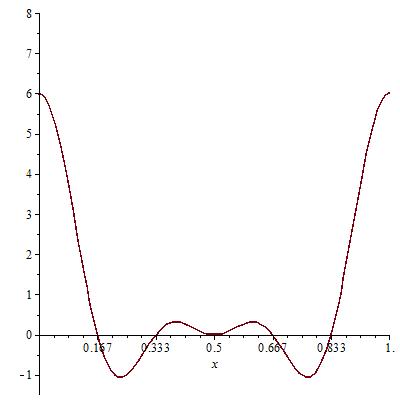}
\caption{de la Vall\'{e}e Poussin kernel $K_{2,2}(x) = V_{2,4}(x)$ on $[0,1]$.}
\label{V_{2,4}}
\end{figure}

This paper studies the $L^1$-norm of 
$V_{rN, sN}(x)$. From \eqref{norm} and \eqref{E:genVdef0} we have the easy bounds
$$
1 \le ||V_{rN, sN}(x)||_{L^{1}(\mathbb T)} = \int_{0}^1 |V_{rN, sN}(x)|  dx \le \frac{s+r}{s-r}.
$$
 Our main result  is the observation that the $L^1$-norms of the kernels in these families 
 are  independent of  the kernel family parameter $N$.

\begin{thm}\label{thm1}
Let $r$ and $s$ be fixed integers
with $0\leq r < s$ and $(r, s)=1$.  Then 
all members of  the kernel family $\sG_{r, s} = \{V_{rN,sN}: N  \ge 1\}$
have the same $L^1$-norm. That is, 
$$
\|V_{rN, sN}\|_{L^{\!1}(\mathbb T)}= \|V_{r, s}\|_{L^{\!1}(\mathbb T)}= \int_{0}^1  |V_{r, s}(x)|dx.
$$
\end{thm}

  This result is surprising because of the  oscillatory nature of these
  functions, which have increasing numbers of sign changes as $N$ increases,  visible in   \figurename~\ref{V_{1,2}}
and  \figurename~\ref{V_{2,4}}; nevertheless 
both functions  have the same $L^1$-norm on $[0,1]$ by Theorem \ref{thm1}.

  For individual values of $r$ and $s$, on taking $N=1$, the value can in principle
  be explicitly determined.  For the special case $(r, s) = (n, n+1)$  we observe that 
  the kernel $V_{n, n+1}(x)$ coincides with the Dirichlet kernel $D_n(x)$ and therefore 
  obtain the following  well known answer.
  
\begin{thm}\label{thm2}
For $n \ge 1$ we have
$$
V_{n, n+1}(x) = D_n(x) = \frac{ \sin \pi(2n+1)x}{\sin \pi x}.
$$
Here $D_n(x)$ is the Dirichlet kernel for period $1$ functions.  In particular, 
$$
\|  V_{n, n+1}\|_{L^{1}(\mathbb T)}=    \| D_n\|_{L^{1}(\mathbb T)}=  L_n,
$$
where $L_n$ is the $n$-th Lebesgue constant.
\end{thm}

The Lebesgue constants $L_n$  solve the extremal problem of giving the supremum of the
$n$-th partial sum $|S_n(f, x)|$ of the Fourier series of $f(x)$ where $f(x)$ is any periodic continuous
function of period $1$ having $|f(x)| \le 1$ everywhere (Lebesgue \cite{Leb1912}, see also Timan \cite[Sects. 4.5, 8.2]{Tim63},
Zygmund \cite[Chap. II, Sect. 12]{Zyg68}.) They are given by
\begin{equation}\label{E:Leb}
L_n = \frac{2}{\pi} \int_{0}^{2 \pi} | \frac{\sin((n+1)/2) \theta}{2 \sin (1/2)\theta}| d \theta.
\end{equation}
It is well known that
\[
L_1= \frac13 + \frac{2\sqrt3}{\pi} = 1.435991\ldots
\]
 and 
 \[
L_2 =  \frac{1}{5}+\frac{\sqrt{10-2\sqrt{5}}}{4\pi}\left(1+3\sqrt{5}\right)=1.642188\dots
 \]
  
Combining these two results we obtain for
the original kernel of de la Vall\'ee Poussin \cite[p. 801 top]{deVP18}
and \cite[Sect. 26]{deVP19}, the following answer.


\begin{cor}\label{cor2}
Let $V_{N,2N}(x) = 2\Delta_{2N}(x) - \Delta_N(x)$ be the de la Vall\'ee Poussin kernel.
 Then
\begin{equation}\label{E:L11}
\int_0^1|V_{N,2N}(x)|\,dx = L_1 = \frac13 + \frac{2\sqrt3}{\pi} = 1.43599112\ldots
\end{equation}
for all $N$.
\end{cor}

We remark that de la Vall\'{e}e Poussin 
introduced the delayed  to study pointwise approximation to Fourier series in his 1919 book
on approximation  \cite[Chap. II, Sec. 27, 29]{deVP19}.
These families of kernels 
form approximate identities, i.e. 
there is a constant $M$ such that for each $N \ge 1$, 
\[
\int_{0}^1 |K_{nN, pN} (x)| dx \le M,
\]
\[
\int_{0}^1 K_{nN, pN} (x) \,dx =1,
\] 
while for each $\delta >0$, 
\[
\lim_{N \to \infty} \int_{ \delta \le |u| \le 1} |K_{nN, pN}(x) | dx =0.
\]
In particular,  for continuous functions $f(x)$
on the torus  ${\mathbb T} = \mathbb{R}/ \mathbb{Z}$ one has pointwise convergence
$$
\lim_{N \to \infty} \sigma_{nN, pN}(f, x) = f(x).  
$$
The issue of how fast the approximations converge to a function in given classes has
been much studied, with work up to the 1950's described in  Timan \cite[Sect. 8.4.4]{Tim63}.
Further work includes Efimov \cite{Ef59}, \cite{Ef60}, Teljakovski \cite{Tel60},
 Dahmen \cite{Da78}, Stechkin \cite{Ste78}, and 
Serdyuk et al. \cite{SO08}, \cite{Ser11}, \cite{SOM12}. 

%
%
\section{Proof of  Theorem \ref{thm1}} \label{S:Proof}

\noindent
Let $r$ and $s$ be integers with $0\le r<s$ and $(r,s)=1$.  The parameters $r, s$ will remain fixed throughout
this section, we simplify the notation by omitting mention of $r$ and $s$ when naming functions.
In particular, $V_N$ means $V_{rN,sN}$.

     From \eqref{E:genVdef2} one obtains using  trigonometric addition formulas\footnote{The addition formulas yield
  $\big(\sin (a+b)x \big)\big(\sin(a-b)x \big)= 
   (\sin ax)^2(\cos bx)^2 - (\sin bx)^2 (\cos ax)^2.$
   The right side then equals  $(\sin ax)^2 - (\sin bx)^2,$
  after adding $(\sin ax)^2 (\sin bx)^2$ to both the opposing terms.}
  that
\begin{equation}\label{E:Vdef4}
V_N(x) = \frac{\big(\sin \pi(s+r)Nx\big)\big(\sin\pi(s-r)Nx\big)}{(s-r)N(\sin\pi x)^2}\,.
\end{equation}
Hence $V_N$ has zeros at points of the form $\{ a/((s+r)N): 0 \le a \le (s+r)N -1\}$ and also at points of the form
$\{b/((s-r)N):  0 \le b \le (s-r)N -1\}$, excluding points where $(s+r)N|a$ or $(s-r)N|b$.  Thus there are $(s+r)N-1$ zeros of the
first kind and $(s-r)N-1$ zeros of the second kind, making a total of $2sN-2$ zeros, counting
multiplicity.  

Since $V_N$ is a trigonometric polynomial of degree $sN-1$, we know that it could
have at most $2sN-2$ zeros.  By taking $a=(s+r)n$ and $b = (s-r)n$ 
We can  see that $V_N$ has
a double zero at $n/N$, for $0<n<N$, taking $a= (s+r)n$ and $b= (s-r)n$ above.  

If $r$ and $s$ are of opposite parity, then $\gcd(s+r,s-r)=1$
implies  there are no other double zeros.  If however $r$ and $s$ are both odd, then $\gcd(s+r,s-r) = 2$,
and by taking $a = n(s+r)/2$, $b = n(s-r)/2$ we find that $V_N$ has double zeros at $n/(2N)$
for $0<n<2N$. 

As an example, in Figure \ref{V_{2,4}} we have $r+s =3$,
and $V_{2,4}(x)$ has  $5$ zeros of the first kind and $1$ zero of the second kind;
the five zeros  of first kind are located at $x= \frac{j}{6}$, $1 \le j \le 5$, and the  zero of second kind at $x =\frac{1}{2}$,
so there is  a double zero at $x = \frac{1}{2}.$

     To compute $\int_0^1|V_N(x)|\,dx$, we break the interval $[0,1]$ into subintervals
running from one simple zero of $V_N$ to the next and then summing the area enclosed between each interval. 
 In order to know what sign to attach to
each interval (to obtain $|V_N(x)|$) we need to consider whether $V_N$ is increasing or decreasing at a simple zero.

\begin{proof}[Proof of Theorem \ref{thm1}]

We define the function $F_N(x)$ by the equality $V_N(x) = F_N(x)\sin (\pi(s+r)Nx)$,
so
$$
F_N(x) = \frac{\sin(\pi(s-r)Nx)}{(s-r)N(\sin\pi x)^2}
$$
using \eqref{E:Vdef4}.
 Then for integer $a$ with $\frac{a}{(s+r)N} \not\in\ZZ$, we have
\[
 V_N'\Big(\frac a{(s+r)N}\Big) = (-1)^a F_N\Big(\frac a{(s+r)N}\Big) \pi(s+r)N\,.
\]
Hence
\begin{align}\label{E:sgnV'a}
\sgn &V_N'\Big(\frac a{(s+r)N}\Big) = (-1)^a\sgn \Big(\sin\frac{\pi(s-r)a}{s+r} \Big)\\
&= (-1)^a\sgn \Big(\sin\Big(\pi a - \frac{2\pi ar}{s+r}\Big)\Big)
= -\sgn \Big( \sin\frac{2\pi ar}{s+r}\Big)\,. \notag
\end{align}
Accordingly, we set
\begin{equation}\label{E:epsilondef}
\varepsilon(a) := \sgn \Big( \sin\frac{2\pi ar}{s+r}\Big) \,.
\end{equation}
We note that
\[
\varepsilon(a+s+r) = \sgn \Big(\sin\frac{2\pi(a+s+r)r}{s+r}\Big) 
=\sgn \Big( \sin\frac{2\pi ar}{s+r}\Big)  = \varepsilon(a).
\]
Hence the values $\varepsilon(a)$ are periodic with period $s+r$.

 We define the function $G_N(x)$  so that $V_N(x) = G_N(x)\sin(\pi(s-r)Nx)$, so
 $$
 G_N(x) = \frac{\sin (\pi(s+r)Nx))}{(s-r)N(\sin\pi x)^2},
 $$
 using \eqref{E:Vdef4}.
   Then for integer $b$ with $\frac{b}{(s-r)N} \notin \ZZ$, one has
\[
 V_N'\Big(\frac b{(s-r)N}\Big) = (-1)^b G_N\Big(\frac b{(s-r)N}\Big) \pi(s-r)N  \,.
\]
Hence
\begin{align}\label{E:sgnV'b}
\sgn &V_N'\Big(\frac b{(s-r)N}\Big) = (-1)^b\sgn \Big(\sin\frac{\pi(s+r)b}{s-r}\Big) \\
&= (-1)^b\sgn \big( \sin\Big(\pi b + \frac{2\pi br}{s-r}\Big)\Big)
= \sgn\Big(\sin\frac{2\pi br}{s-r}\Big)\,. \notag
\end{align}
Accordingly, we set
\begin{equation}\label{E:deltadef}
\delta(b) := -\sgn \Big(\sin\frac{2\pi br}{s-r}\Big) \,.
\end{equation}
We note that
\[
\delta(b+s-r) = -\sgn \Big( \sin\frac{2\pi(b+s-r)r}{s-r}\Big)
=-\sgn \Big(\sin\frac{2\pi br}{s-r} \Big)= \delta(b)\,.
\]
Hence the values $\delta(b)$ are periodic with period $s-r$.

   Next set
\begin{equation}\label{E:WNdef}
W_N(x) := x + \sum_{n\ne0}\frac{v_{r,s-r}(n/N)}{2\pi i n}e(nx),
\end{equation}
and note  that $W_N^{'}(x) = V_N(x)$ by \eqref{E:genVdef3}.  Suppose that $x_{k-1}, x_k, x_{k+1}$ are three consecutive simple
zeros of $V_N$ in $(0,1)$, and suppose that $x_k = a/((s+r)N)$. If $V_N'(x_k) < 0$, then
$V_N(x)>0$ for $x_{k-1}<x<x_k$ and $V_N(x) <0$ for $x_k< x < x_{k+1}$.  These intervals contribute
to the integral an amount
\begin{align*}
\big(W_N(x_k) &- W_N(x_{k-1})\big) - (W_N(x_{k+1}) - W_N(x_k)\big) \\
&= -W_N(x_{k-1}) +2W_N(x_k) - W_N(x_{k+1})\,.
\end{align*}
In this situation $\varepsilon(a) = 1$, so the point $x_k$ contributes
$2\varepsilon(a)W_N(x_k)$. If $V_N'(x_k)>0$, then all signs are reversed, and the contribution
of $x_k$ is still $2\varepsilon(a)W_N(x_k)$.  Now suppose that $x_k$ is a double zero, and
that $V_N(x)>0$ in the two intervals.  Then these intervals contribute
\[
\big(W_N(x_k)-W_N(x_{k-1})\big) + \big(W_N(x_{k+1})-W_N(x_k)\big)
= W_N(x_{k+1}) - W_N(x_{k-1})\,.
\]
In this case, $x_k$ makes no contribution, but $\varepsilon(a)=0$, so the contribution is still
$2\varepsilon(a)W_N(x_k)$.  If $V_N(x)<0$ in these two intervals then all signs are reversed, but
the contribution of $x_k$ is still $2\varepsilon(a)W_N(x_k)$.  Similarly, if
$x_k = b/((s-r)N)$, then the contribution of $x_k$ is $2\delta(b)W_N(x_k)$.

   The interval $[0, 1/((s+r)N)]$
contributes $W_N(1/((s+r)N)) - W_N(0)$. The first term here is half of the contribution made
by the point $1/((s+r)N)$, since $\varepsilon(1) = 1$.  The contribution made by the interval
$[1-1/((s+r)N), 1]$ is $W_N(1)-W_N(1-1/((s+r)N))$.  The latter term is half the contribution made by
the point $1-1/((s+r)N)$, since $\varepsilon((s+r)N-1) = -1$. We note that $W_N(1)-W_N(0) = 1$.
Hence we conclude that
\begin{align}\label{E:L12}
\int_0^1|V_N(x)|\,dx = 1 &+ 2\sum_{a=1}^{(s+r)N}\varepsilon(a)W_N\Big(\frac a{(s+r)N}\Big) \\
&+ 2\sum_{b=1}^{(s-r)N}\delta(b)W_N\Big(\frac b{(s-r)N}\Big)\,.\notag
\end{align}
Here the terms $a=(r+s)N$ and $b = (s-r)N$ ought not to be included in the above, since $V_N(1)\ne 0$.
However, $\varepsilon((s+r)N)=0$ and $\delta((s-r)N) = 0$, so no harm is done.

    We write $W_N(x) = x + X_N(x)$ and evaluate the contributions of the two terms separately
    to the right side of \eqref{E:L12}.  The contribution of the linear term $x$ to the sum \eqref{E:L12}  is
\[
2\sum_{a=1}^{(s+r)N}\frac{\varepsilon(a)a}{(s+r)N} + 2\sum_{b=1}^{(s-r)N}\frac{\delta(b)b}{(s-r)N}\,.
\]
Since the $\varepsilon(a)$ are periodic with period $s+r$, and the $\delta(b)$ are periodic with
period $s-r$, the above is
\[
= 2\sum_{n=0}^{N-1}\bigg(\sum_{a=1}^{s+r}\varepsilon(a)\Big(\frac a{(s+r)N}+\frac nN\Big)
+\sum_{b=1}^{s-r}\delta(b)\Big(\frac b{(s-r)N} + \frac nN\Big)\bigg)\,.
\]
On reversing the order of the double sums, we see that this is
\begin{equation}\label{E:xcon1}
= 2\sum_{a=1}^{s+r}\varepsilon(a)\Big(\frac a{s+r} + N-1\Big)
+2\sum_{b=1}^{s-r}\delta(b)\Big(\frac b{s-r} + N - 1\Big)\,.
\end{equation}
Since $V_N(0) = V_N(1) = s+r > 0$, in the interval $[0,1]$ we pass from positive values to negative
 the same number of times that we pass from negative values to positive.  That is,
\[
0=\sum_{a=1}^{(s+r)N}\varepsilon(a) + \sum_{b=1}^{(s-r)N}\delta(b)
=N\sum_{a=1}^{s+r}\varepsilon(a) +N\sum_{b=1}^{s-r}\delta(b)\,.
\]
Hence the expression \eqref{E:xcon1} is
\begin{equation}\label{E:xcon2}
=\frac2{s+r}\sum_{a=1}^{s+r}\varepsilon(a)a + \frac2{s-r}\sum_{b=1}^{s-r}\delta(b)b\,.
\end{equation}

     The contribution of $X_N(x)$ to the right side of \eqref{E:L12}  is
\begin{equation}\label{E:XNcon1}
2\sum_{a=1}^{(s+r)N}\varepsilon(a)X_N\Big(\frac a{(s+r)N}\Big)
+ 2\sum_{b=1}^{(s-r)N}\delta(b)X_N\Big(\frac b{(s-r)N}\Big)\,.
\end{equation}
Here the sum over $a$ is
\begin{equation}\label{E:asumcon1}
\sum_{n\ne 0}\frac{v_{r,s-r}(n/N)}{\pi in}\sum_{a=1}^{(s+r)N}\varepsilon(a)e\Big(\frac{an}{(s+r)N}\Big)\,.
\end{equation}
where $v_{r,s-r}(n/N)$ are given in \eqref{E:vdef}. 
Since the $\varepsilon(a)$ have period $s+r$, we know by the theory of the Discrete Fourier Transform 
 that there exist numbers $\widehat{\varepsilon}(k)$ such that
\begin{equation}\label{E:epsilonDFT}
\varepsilon(a) = \sum_{k=1}^{s+r}\widehat{\varepsilon}(k)e\Big(\frac{ak}{s+r}\Big)
\end{equation}
holds for all integer $a$.  Hence the expression \eqref{E:asumcon1} is
\[
= \sum_{n\ne 0}\frac{v_{r,s-r}(n/N)}{\pi in}\sum_{k=1}^{s+r}\widehat{\varepsilon}(k)
\sum_{a=1}^{(s+r)N}e\left(\frac{a(n+kN)}{(s+r)N}\right).
\]
Here the innermost sum is $(s+r)N$ if $n\equiv-kN\pmod{(s+r)N}$, and is $0$ otherwise.  We write
$n = -kN +m(r+s)N$.  Then the above is
\begin{equation}\label{E:asumcon2}
=(s+r)\sum_{k=1}^{s+r}\widehat{\varepsilon}(k)\sum_{\substack{m \in \ZZ \\ (s+r)m\ne k}}
\frac{v_{r,s-r}(-k+m(s+r))}{\pi i(-k+m(s+r))}\,.
\end{equation}
We note that if $1\le k\le s+r$, then $v(-k+m(s+r))=0$ if $m > 1$ or if $m <0$.  Thus the sum
over $m$ can be restricted to just $m = 0, 1$.  However, the main point of the above is that
it is independent of $N$.

     The sum over $b$ in \eqref{E:XNcon1} is
\begin{equation}\label{E:bsumcon1}
\sum_{n\ne 0}\frac{v_{r,s-r}(n/N)}{\pi in}\sum_{b=1}^{(s-r)N}\delta(b)e\Big(\frac{bn}{(s-r)N}\Big)\,.
\end{equation}
The $\delta(b)$ have period $s-r$, so let numbers $\widehat{\delta}(n)$ be determined so that
\begin{equation}\label{E:deltaDFT}
\delta(b) = \sum_{k=1}^{s-r}\widehat{\delta}(k)e\Big(\frac{kb}{s-r}\Big)
\end{equation}
for all $b$.  
Hence the expression \eqref{E:bsumcon1} is
\[
= \sum_{n\ne 0}\frac{v_{r,s-r}(n/N)}{\pi in}\sum_{k=1}^{s-r}\widehat{\delta}(k)
\sum_{b=1}^{(s-r)N} e\Big(\frac{b(n+kN)}{(s-r)N}\Big)\,.
\]
The innermost sum is $(s-r)N$ if $n\equiv-kN\pmod{(s-r)N}$, and is $0$ otherwise.  We write
$n = -kN + m(s-r)N$.  Then the above is
\begin{equation}\label{E:bsumcon2}
(s-r)\sum_{k=1}^{s-r}\widehat{\delta}(k)\sum_{\substack{m \in \ZZ\\(s-r)m\ne k}}\frac{v_{r,s-r}(-k+m(s-r))}{\pi i(-k+m(s-r))}\,.
\end{equation}
This formula and \eqref{E:asumcon2} serve to evaluate the expression \eqref{E:XNcon1}.  On combining
this evaluation with \eqref{E:xcon2} in \eqref{E:L12}, we conclude that
\begin{align}\label{E:L1final}
\int_0^1|V_N(x)|\,dx &= 1 + \frac2{s+r}\sum_{a=1}^{s+r}\varepsilon(a)a
+\frac2{s-r}\sum_{b=1}^{s-r}\delta(b)b \notag \\
&\quad+(s+r)\sum_{k=1}^{s+r}\widehat{\varepsilon}(k)\sum_{\substack{m \\ (s+r)m\ne k}}
\frac{v_{r,s-r}(-k+m(s+r))}{\pi i(-k+m(s+r))}  \\
&\quad +(s-r)\sum_{k=1}^{s-r}\widehat{\delta}(k)\sum_{\substack{m\\(s-r)m\ne k}}
\frac{v_{r,s-r}(-k+m(s-r))}{\pi i(-k+m(s-r))}\, \notag.
\end{align}
Since this value is independent of $N$, the proof is complete.

\end{proof}


\section{Proofs of  Theorem \ref{thm2} and Corollary \ref{cor2}}\label{S:coro}

\begin{proof}[Proof of Theorem \ref{thm2}].
Recall from \eqref{E:Vdef4} that  for $(r, s)=1$, 
\begin{equation}
V_{r, s}(x) = \frac{\big(\sin \pi(s+r)x\big)\big(\sin\pi(s-r)x\big)}{(s-r)(\sin\pi x)^2}\,.
\end{equation}
In the special case $s-r=1$, which corresponds to $(r, s)= (n, n+1)$ we obtain the simplification
\[
V_{n, n+1}(x) = \frac{\big(\sin \pi(2n+1)x\big)\big(\sin\pi x\big)}{(\sin\pi x)^2} = \frac{\sin \pi(2n+1)x}{\sin \pi x}= D_{n}(x) 
\]
The right hand side is exactly the Dirichlet kernel $D_{n}(x)$, rescaled to the interval $[0,1]$.
By definition the Lebesgue constant 
$$
L_n = \| D_n(x)\|_{L^1({\mathbb T}) }= \int_{0}^1 |D_n(x)| dx,
$$
which on rescaling to the usual interval $[0, 2\pi]$ recovers the usual definition \eqref{E:Leb}.
\end{proof}

We give explicit computations yielding Corollary \ref{cor2}.

\begin{proof}[Proof of Corollary \ref{cor2}.]
The result follows on combining Theorems \ref{thm1} and \ref{thm2}.
For the explicit value we have 
\[
V_{1,2}(x) = 1 + 2\cos2\pi x = \frac{\sin3\pi x}{\sin\pi x}
\]
by \eqref{E:Vdef3} and \eqref{E:Vdef4}.  Hence
\begin{align*}
\int_0^1 |V_{1,2}(x)|\,dx &=\Big[x+\frac{\sin2\pi x}\pi\Big|_0^{1/3}
-\Big[x+\frac{\sin2\pi x}\pi\Big|_{1/3}^{2/3} +\Big[x+\frac{\sin2\pi x}\pi\Big|_{2/3}^1 \\
&=\frac13+\frac{2\sqrt3}\pi\,.
\end{align*}

     Alternatively, one can argue from \eqref{E:L1final}.  We find that $\varepsilon(1)=1$,
$\varepsilon(2)=-1$, $\varepsilon(3) = 0$, $\widehat{\varepsilon}(1)=-i/\sqrt3$,
$\widehat{\varepsilon}(2)=i/\sqrt3$, $\widehat{\varepsilon}(3)=0$, and
$\delta(1) = \widehat{\delta}(1)= 0$.  The result is the same.
\end{proof}

\section{Concluding Remarks }\label{further}
We showed  that the $L^1$-norm of 
\[V_{rN,sN}(x)=\frac{(\sin{\pi(s+r)Nx})\, (\sin{\pi(s-r)Nx})}{(s-r)N\sin^2{\pi x}}\]
is independent of $N$ . If we let $A_{r, s, N}^{+}$ resp. $A_{r, s, N}^{-}$ denote the
positive and negative areas of the graph then we have shown
$$
A_{r, s, N}^{+} - A^-_{r, s, N} = \| V_{r,s}\|_{L^1({\mathbb T})}
$$
is independent of $N$. Since $A_{r, s, N}^{+} + A_{r, s, N}^{-} = 1$
we have that both these areas are independent of $N$, with
$$
A_{r, s, N}^{+} = \frac{1}{2} \Big( 1+  \| V_{r,s}\|_{L^1({\mathbb T})}\Big).
$$
The effect of increasing $N$ is does not change the area, but  shifts  its location. 
As $N$ increases most area (both positive and negative)
is concentrated near integer values of $x$. One can  show that
$$
|V_{rN, sN}(x)| \le \frac{4}{\sqrt[3]{N}} ~~~\mbox{for}  ~~\frac{1}{\sqrt[3]{N}} \le x \le 1- \frac{1}{\sqrt[3]{N}}.
$$

\section{Acknowledgments}\label{sec5}
The author thanks H. L. Montgomery for suggesting the project of
improving the bounds for   $L^1$ norms
of de la Vallee Poussin kernels as part of an REU program at
the University of Michigan in summer 2012. He  thanks H. L. Montgomery and J. C. Lagarias for  references and for 
editorial assistance with exposition. 
The author  thanks P. Nevai for helpful comments and references. P. Nevai  observed that
  Theorem \ref{thm1} can be  deduced from results  which 
are contained in his unpublished 1969 manuscript (Nevai \cite[Lemma 3]{Nev69}).

\nocite{*}
\bibliographystyle{cdraifplain}

\end{document}